\documentclass[11pt,a4paper]{amsart}%
\usepackage{cite}
\usepackage[citecolor=red]{hyperref}
\usepackage{amssymb}
\usepackage{amsmath}
\usepackage{amsthm}
\usepackage{amsfonts}
\usepackage{graphicx}
\usepackage{bbm}
\usepackage{color}
\usepackage[toc,page]{appendix}
\usepackage{subfigure}
\usepackage{tikz, tikz-cd}

\newcommand{\beq}{\begin{equation}}
\newcommand{\eeq}{\end{equation}}
\newcommand{\ba}{\begin{align*}}
\newcommand{\ea}{\end{align*}}
\newcommand{\bea}{\begin{eqnarray}}

\newcommand{\eea}{\end{eqnarray}}
\newcommand{\beas}{\begin{eqnarray*}}
\newcommand{\eeas}{\end{eqnarray*}}

\newtheorem{thm}{Theorem}[section]
\newtheorem{cor}[thm]{Corollary}

\newtheorem{lem}[thm]{Lemma}
\newtheorem{prop}[thm]{Proposition}

\theoremstyle{definition}

\newtheorem{example}{Example}[section]
\newtheorem{rem}[thm]{Remark}

\numberwithin{equation}{section}

\newcommand{\ricci}{\mathbf{Ric}}

\newcommand{\pt}{\frac{\partial}{\partial t}}
\newcommand{\cL}{\mathcal L}

\newcommand{\var}{\varepsilon}

\newcommand{\cF}{\mathcal F}

\newcommand{\M}{\mathbb M}
\newcommand{\B}{\mathbb B}
\newcommand{\cH}{\mathcal H}
\newcommand{\cV}{\mathcal V}

\newcommand{\R}{\mathbb R}

\begin{document}

\title[Harnack inequalities with transverse Ricci flow]{Harnack inequalities on totally geodesic foliations with transverse Ricci flow.}

\author[Feng, Qi]{Qi Feng{$^{\dag}$}}
\thanks{\footnotemark {$\dag$} Research was supported in part by NSF Grant  DMS-1660031.(F. Baudoin PI)}
\address{$^{\dag}$ Department of Mathematics\\
University of Connecticut\\
Storrs, CT 06269,  U.S.A.}
\email{qi.feng@uconn.edu}



\date{\today }
\maketitle
\begin{abstract}
In the current paper,under the transverse Ricci flow on a totally geodesic Riemannian foliation, we prove two types of differential Harnack inequalities (Li-Yau gradient estimate) for the positive solutions of the heat equation associated with the time dependent horizontal Laplacian operators. We also get a time dependent version of the generalized curvature dimension inequality. As a consequence of aforementioned results, we also get parabolic Harnack inequalities and heat kernel upper bounds.
\end{abstract}
\tableofcontents

\section{Introduction}

In general, Ricci curvature, which is defined as the trace of the Riemannian curvature tensor, plays an important role in the study of Riemannian geometry. Recently, there are two other well known methods to define Ricci curvature lower bounds. One is by Bakry-\'Emery \cite{BakryEmery85} to use curvature dimension inequality to characterize Ricci curvature lower bound. The other one is by Sturm-Lott-Villani \cite{LottVillani09, Sturm06} to define the Ricci curvature lower bound through optimal transportation. These two notions of Ricci curvature are then extended to metric measure spaces and  are shown to be equivalent in certain cases \cite{AGS15}. Moreover, recently Naber-Haslhofer \cite{Naber13,HaslhoferNaber16ricci} characterized Ricci curvature two sided bound by using stochastic analysis on the path space of a Riemannian manifold.  The sub-Riemannian analogue of the curvature dimension inequality is the generalized curvature dimension inequality introduced by Baudoin-Garofalo \cite{BaudoinGarofalo14}. They proved the Li-Yau \cite{YauLi86} type inequality on sub-Riemannian manifolds by using the generalized curvature dimension inequality as well as other inequalities and heat kernel bounds. But the sub-Riemannian analogue of Sturm-Lott-Villani \cite{LottVillani09, Sturm06} is not true even for the simple case like Heisenberg group \cite{Juillet09}. However, the characterization of two sided Ricci curvature bounds through stochastic analysis on the path space of a totally geodesic foliation is true \cite{BaudoinFeng15}.

Another very important method to study Riemannian geometry is the Ricci flow introduced by Hamilton \cite{Hamilton82}, for a detailed study of Ricci flow we refer to \cite{ChowKnopf04}. The study of Ricci flow on a Riemannian manifold under the curvature dimension inequality
condition is recently studied  by Li-Li \cite{LiLi17}. The characterization of Ricci flow by using optimal transportation is carried out by McCann-Topping \cite{MccannTopping10}. Recently, Naber-Haslhofer-Hein \cite{HaslhoferNaber15weak, HeinNaber14} try to define the weak solution of Ricci flow through the analysis (and stochastic analysis) on the path space of a Riemannian manifold. The sub-Riemannian analogue of these aforementioned works are of particular interest. 

The first paper about Ricci flow on sub-Riemannian manifold is studied by Lovric-Min-Oo-Ruh \cite{LMR00}. They showed the existence of the Ricci flow on Riemannian foliaitons by flowing the Cartan connection. There are several recent study of Ricci flow on Riemannian foliations. For example:  mixed curvature Ricci flow on co-dimension one foliations  \cite{Rovenski13, Rovenski13partial}, Sasakian transverse Ricci flow \cite{Collins12, SWZ10}, general second order geometric flows on Riemannian foliations \cite{BHV15}. But in all these previous work,  they all have torsion free condition for the transverse Levi-Civita connection. The first non-torsion free case in the context of Ricci flow is the work by Phong-Picard-Zhang \cite{PhongPicardZhang15, PhongPicardZhang16} where they study anomaly flow with Strominger system and Fu-Yau equation on 3-dimensional complex manifolds. 

On the other side, sub-Riemannian geometry has drawn significant attention due to its advances connections to active research areas like stochastic analysis on manifolds \cite{BaudoinFeng15,BaudoinFengGordina17, GrongThalmaier16stochastic}, optimal transportation \cite{FigalliRifford10, KhesinLee09}, etc. In particular, the recent work about the generalized curvature dimension inequality \cite{BaudoinGarofalo14,GrongThalmaier16curvature1,GrongThalmaier16curvature2,BKW15} and sub-Laplacian comparison theorems \cite{BGKT17,AgrachevLee15} offer us powerful tools to conduct new research in sub-Riemannian geometry.

In this paper, we try to connect the two directions on Ricci flow and the generalized curvature dimension inequalities on sub-Riemannian manifolds, in particular on totally geodesic foliations. Due to the length of our paper, the paper will be divided into two parts. In this first paper, we first start to study Ricci flow under the non-torsion free condition on totally geodesic Riemannian foliations. We prove that the sub-Riemannian structure is preserved under this flow and prove various differential inequalities under this context. In the second paper \cite{Feng17}, we will show the long time existence of the Ricci flow, curvature bound estimates, monotonicity formulas, Perelman's non-collapsing condition and rigidity models in a sub-Riemannian setting. It is worth to point it out that the inequalities we proved in the current paper will play an important role in our second paper \cite{Feng17}.


In a series of future papers, the author is extremely interested to study in a more general sub-Riemannian setting of the connections among Ricci flow, generalized curvature dimension inequality, stochastic analysis and optimal transportation on the path space of a sub-Riemannian manifold.  The goal of developing aforementioned work is hoping to solve some \textbf{Open Problems} in sub-Riemannian geometry. One open problem was posted by A.V. Toponogov in 80's which is stated in \cite{Rovenski13partial}, see details therein. 

\textbf{Open Problem 1:} Given a Riemannian manifold $(\M^{m+n},g)$ foliated with complete totally geodesic leaves of dimension $m$. Assume that the sectional curvature of $\M$ are positive for all planes of two vectors such that the first (second) vector is tangent (orthogonal) to a leave. Then 
\begin{align}\label{Ferus estimate}
m\le \rho(n)-1
\end{align}
where $\rho(n)-1$ is the maximal number of point-wise linear independent vector fields on a sphere.
The estimate \eqref{Ferus estimate} is proved in \cite{Ferus70, Rovenski98} in some special case. There is another interesting estimate \cite{BaudoinGarofalo09} relating the Hausdorff dimension of $\M^{n+m}$, i.e. $dim_{Haus}(\M)\le D$, where $D$ is a constant related to the curvature bounds in the generalized curvature dimension inequality \cite{BaudoinGarofalo14}.

\textbf{Open Problem 2:} What is the relation between \eqref{Ferus estimate} and D? 

 We think that the improvements of this constant $D$ may give us sharp diameter estimate for \emph{Bonnet-Myer's} theorem in sub-Riemannian geometry.  
 
 \textbf{Open Problem 3:} What is the rigidity model for Perelman's \cite{Perelman02} type $\mathcal W$ entropy functional in sub-Riemannian geometry?
 A good candidate for such a rigidity model might be a Hopf fibration or quaternionic fibration.\\

 
In this paper we consider $\M$ as complete general Riemannian foliations with totally geodesic leaves \cite{Baudoin14} and bundle like metric \cite{Reinhart59}. A natural connection come with this Riemannian foliation structure is the so-called $Bott$ connection \cite{Besse07}, which coincides with the transverse $Levi$-$Civita$ connection acting on basic functions. We will study a pair of equations which are the transverse Ricci flow for the $Bott$ connection
   \beq\label{transverse RF}
 \begin{cases}
 & \frac{\partial g_{\cH}}{\partial t}=-2\ricci_{\cH},\\
 &\frac{\partial g_{\cV}}{\partial t}=0,
 \end{cases}
 \eeq
 and the heat equation associated with the time dependent horizontal Laplacian operator $L^t$ $(=\Delta_{g_{\cH}(t)})$ for metric $g(t)=g_{\cH}(t)\oplus g_{\cV}.$
 \begin{align}\label{heat equation}
 (L^t-\pt)u(x,t)=0,\quad x\in\M, t\in[0,T].
 \end{align}

   In section 2, we first introduce the basic notations and backgrounds on Riemannian foliations.    
   
   In section 3, we state our main results, comment on their relations and applications, as well as comparisons to previous work.

   In section 4, we first introduce the general transverse geometric flow
   \begin{align}\label{general RF 1}
\begin{cases}
&\pt g_{\cH}(t)=h(t),\quad g_{\var}(0)=g_{\var,0}=g_{\cH,0}\oplus \frac{1}{\var}g_{\cV,0},\\
&\pt g_{\cV}=0.
\end{cases}
\end{align}
   Under our assumptions on $h(t)$, we show that under our general transverse flow \eqref{general RF 1}, we can find local time dependent orthonormal frames and the Riemannian foliation structure with totally geodesic leaves and bundle like metric are preserved. In particular, if the Riemannian foliation satisfies the \emph{Yang-Mills} condition, the flow \eqref{general RF 1} will also preserve this property. Moreover, all these properties are also preserved under our transverse Ricci flow \eqref{transverse RF}, since $h=-2\ricci_{\cH}$ is a special case of flow \eqref{general RF 1}.

    In section 5, by using the local time dependent orthonormal frames on the manifold, we can deduce a time-dependent version of the Bochner-Weitzenb\"ock identity which gives us a time dependent version of the generalized curvature dimension inequality. In this section, we represent the positive solution to the heat equation \eqref{heat equation} at time $t$  in terms of the semigroup $P_{s,t}f$ with initial value $f$ at time $s$. By using this representation, we generalize the Li-Yau inequality in the generalized curvature dimension equality setting by Baudoin-Garofalo \cite{BaudoinGarofalo14} to our time dependent setting. As a consequence, we can get the parabolic Harnack inequality and heat kernel upper bound.
   
    In section 6, we prove another Li-Yau inequality which generalizes the result by Bailesteanu-Cao-Pulemotov \cite{BCP10} on a Riemannian manifold to our Riemannian foliation setting. Since the semigroup representation for the solution to the heat equation associated to the time dependent horizontal Laplacian operator do not clearly show the fact that the heat on the manifold spreads over the manifold and the manifold itself evolves in time at the same time, we thus directly look at  the solution $u(x,t)$ at time $t$. In order to get Li-Yau type inequality (as well as parabolic Harnack inequality) for the positive solution to the heat equation associated with the time dependent horizontal Laplacian operator, we combine the methods by Bailesteanu-Cao-Pulemotov \cite{BCP10} and Qian \cite{Qian13}. Eventually, the bound for our differential Harnack inequality will depends on a time dependent function $\alpha(t)$ and other geometric bounds on curvatures, which is similar to Qian \cite{Qian13} where he proved differential Harnack inequality for the positive solutions to the Sch\"odinger equation associated to the subelliptic operator with potential. By choosing a good candidate 
for the function $\alpha(t)$, we are able to get Li-Yau inequality and parabolic Harnack inequality which are comparable to the previous semigroup setting. In particular, our result generalize the space-time gradient estimates under Ricci flow on a Riemannian manifold in \cite{BCP10}  and what's more, our result actually works for transverse super Ricci flow. The author believes that similar results should also hold under transverse $k$-super Ricci flow, though this requires more work.

\section{Notation}
 \subsection{Riemannian foliation}
Let $\mathbb{M}$ be a smooth connected  and complete manifold of the dimension $n+m$. Assume that $\mathbb{M}$ is equipped with a Riemannian foliation structure, $\mathcal{F}$,  with a bundle-like   metric $g$ and totally geodesic $m$-dimensional leaves. We denote $(\M,\mathcal F, g)$ as such a Riemannian foliation and for details we refer to \cite{Baudoin14, Molino88, Reinhart59, Tondeur88}.

Denote subbundle $\mathcal{V}$ as the set of \emph{vertical directions}  formed by the vectors tangent to the leaves and denote the subbundle $\mathcal{H}$ as the set of \emph{horizontal directions} which is normal to $\mathcal{V}$. We assume that the Lie algebra of vector fields generated by global $C^{\infty}$ sections of $\mathcal{H}$ has the full rank at each point in $\M$, which means that $\mathcal{H}$ satisfies the bracket generating condition.

 In this paper, denote $T\M$ as the \emph{tangent bundle} and $T^{\ast}\M$ as the \emph{cotangent bundle}. In particular, we denote $T_{x}\M$ ($T_{x}^{\ast}\M$) as the \emph{tangent (cotangent) space} at $x \in \M$. Denote $g\left( \cdot, \cdot \right)$,
 $g_{\mathcal{H}}\left( \cdot, \cdot \right)$, $g_{\mathcal{V}}\left( \cdot, \cdot \right)$ as the inner product on $T\M$ induced by the metric $g$ and its restrictions to $\mathcal{H}$ and $\mathcal{V}$ respectively. As always, for any $x \in \M$ denote by $g\left( \cdot, \cdot \right)_{x}$ (or $\langle \cdot, \cdot \rangle_x$), $g_{\mathcal{H}}\left( \cdot, \cdot \right)_{x}$ (or $\langle \cdot, \cdot \rangle_{\mathcal{H}_x}$), $g_{\mathcal{V}}\left( \cdot, \cdot \right)_{x}$  (or $\langle \cdot, \cdot \rangle_{\mathcal{V}_x}$) the inner product on the fibers $T_{x}\M$, $\mathcal{H}_{x}$ and $\mathcal{V}_{x}$ correspondingly. Let $\mathcal C^\infty(\M)$ denote as the space of \emph{smooth functions} on $\M$ and $\mathcal C_0^\infty(\M)$ denote as the space of \emph{smooth and compactly supported functions}. Let $\Gamma^\infty(\mathcal{E})$ denote as the space of \emph{smooth sections} of a vector bundle $\mathcal{E}$ over $\M$ and $\Gamma_0^\infty(\mathcal{E})$ denote as the space of \emph{smooth and compactly supported sections}.
\subsection{Bott connection}
On the Riemannian manifold $(\M,g)$ there is the Levi-Civita connection that we denote by $\nabla^R$, but the suitable connection which is adapted to our study of foliations is the \emph{Bott connection} on $\mathbb{M}$. Define the \emph{Bott connection} \cite{Besse07} in the following way,
\begin{align}\label{Bott connection}
\nabla_X Y =
\begin{cases}
\pi_{\mathcal{H}} ( \nabla_X^R Y), X, Y \in \Gamma^\infty(\mathcal{H}),
\\
\pi_{\mathcal{H}} ( [X,Y]),  X \in \Gamma^\infty(\mathcal{V}), Y \in \Gamma^\infty(\mathcal{H}),
\\
\pi_{\mathcal{V}} ( [X,Y]),  X \in \Gamma^\infty(\mathcal{H}), Y \in \Gamma^\infty(\mathcal{V}),
\\
\pi_{\mathcal{V}} ( \nabla_X^R Y), X, Y \in \Gamma^\infty(\mathcal{V}),
\end{cases}
\end{align}
where $\pi_\mathcal{H}$ (resp. $\pi_\mathcal{V}$) is the projection on $\mathcal{H}$ (resp. $\mathcal{V}$). It is easy to check that  the \emph{Bott} connection is metric-compatible, that is, $\nabla g=0$, and it is not torsion-free. We denote $T$ to be the torsion of the $Bott$ connection $\nabla$. For more properties of the torsion $T$, we refer to \cite{BaudoinFengGordina17}.

For any bundle like metric (see \cite{Reinhart59}) $g$, we can always have a orthogonal decomposition as $g=g_{\cH}\oplus g_{\cV}.$ 
In this paper, we will consider a family of one parameter variation of the metric $g$ as 
\[
g_{\var}=g_{\cH}\oplus \frac{1}{\var}g_{\cV}\quad \text{where} ~\var>0.
\]
One can check that $\nabla g_{\var}=0$ for every $\var>0$. The metric $g_{\var}$ introduces a metric on the cotangent  bundle which we still denote by $g_{\var}$. By using the similar notations as before we have 
\[
\|\eta\|_{\var}^2=\|\eta\|_{\cH}^2+\var\|\eta\|_{\cV}^2,\quad \forall ~~\eta\in T^*\M.
\]
We say that a one-form is horizontal (resp. vertical) if it vanishes on the vertical bundle $\cV$ (resp. on $\cH$).
We define our horizontal Ricci curvature $\mathfrak{Ric}_{\mathcal{H}}$(or $\ricci_{g_{\cH}(t)}$) of $Bott$ connection as the fiberwise symmetric linear map on one-forms such that for all smooth functions $f, g$ on $\M$
\[
\langle  \mathfrak{Ric}_{\mathcal{H}} (df), dg \rangle=\mathbf{Ric} \left(\nabla_\mathcal{H} f, \nabla_\mathcal{H} g\right)=\mathbf{Ric}_\mathcal{H} ( \nabla f, \nabla g),
\]
where $\mathbf{Ric}$ is the Ricci curvature of the $Bott$ connection and $\mathbf{Ric}_\mathcal{H}$ its horizontal Ricci curvature (horizontal trace of the full curvature tensor $R$  of the $Bott$ connection ). For each $Z \in \Gamma^\infty(\mathcal{V})$ there is a unique skew-symmetric endomorphism  $ J_Z: \mathcal{H}_x \to \mathcal{H}_x$, $x \in \M$ such that for all horizontal vector fields $X, Y \in \mathcal{H}_{x}$
\begin{align}\label{J map}
g_\mathcal{H} ( J_Z (X), Y)_{x}= g_\mathcal{V} (Z, T(X, Y))_{x},
\end{align}
where $T$ is the torsion tensor of $\nabla$ and we extend $J_{Z}$ to be $0$ on  $\mathcal{V}_x$. We now introduce the following so-called damped connections (see \cite{Baudoin14}) $\nabla^{\var}$ as
\[
\nabla^{\var}_XY=\nabla_XY-\mathfrak T^{\var}_XY=\nabla_XY-T(X,Y)+\frac{1}{\var}J_YX, \quad X,Y\in\Gamma^{\infty}(\M).
\]
In this paper, our metric $g(t)$ evolves in time, according to Lemma \ref{ON frame}, we know that there exists orthonormal frames depending on time $t$. We denote $\{X_1(t),\cdots,X_n(t)\}$ as the horizontal time dependent orthonormal frames and $\{Z_1,\cdots,Z_m\}$ as the vertical time independent orthonormal frames (for details, we refer to Lemma \ref{ON frame} in section $4$). In particular, we can represent a generic one-form as $\eta=\sum_{i=1}^nf_i\theta_i(t)+\sum_{j=1}^mk_j v_j$,
where $\{ \theta_1(t),\cdots,\theta_n(t),v_1,\cdots,v_m\}$ is the dual coframe of $\{X_1(t),\cdots,X_n(t),Z_1,\cdots,Z_m \}$

Then for metric $g(t)$ and the associated $Bott$ connection $\nabla^t$, we have the following time dependent horizontal laplacian operator
 \[
 L^t=-\nabla^*_{g_{\cH}(t)}\nabla_{g_{\cH}(t)}(\text{or}=-\nabla_{\cH}^*\nabla_{\cH}), 
 \]
 by using the local coordinates \ref{ON frame}, we can represent $L^t$ by
 \[
 L^t=\sum_{i=1}^n\nabla^t_{X_i(t)}\nabla^t_{X_i(t)}-\nabla^t_{\nabla^t_{X_i(t)}X_i(t)}.
 \]
 Recall our definition of $J$. If $Z_1,\cdots,Z_m$ is a local vertical frame, then $J$ define a $(1,1)$ tensor 
\[
\mathbf J^2:=\sum_{j=1}^m J_{Z_j}J_{Z_j}.
\]
This does not depend on the choice of the frame and may be defined globally. The horizontal divergence of the torsion $T$ is also a $(1,1)$ tensor which in a local horizontal frame $\{X_1(t),\cdots,X_n(t)\}$ is defined as
\[
\delta_{\cH}T(X):=-\sum_{i=1}^n(\nabla^t_{X_i(t)}T)(X_i(t),X).
\]
In particular, for a generic one form $\eta=\sum_{i=1}^nf_i\theta_i(t)+\sum_{j=1}^mk_j v_j$, we have
\[
\delta_{\cH}T(\eta)=\sum_{i,j=1}^n\sum_{l=1}^m (X_i(t)\gamma_{ij}^l(t))f_jv_l,
\]
where $\gamma_{ij}^l(t)$ is from the structure equation in lemma \ref{ON frame}, a static version of this expression can be found in \cite{BKW15}.
In particular, if $\delta_{\cH}T(\cdot)=0$, we say that it satisfies the \emph{Yang-Mills} condition.
With all the above definition in hand, we are ready to introduce the following operator
 \[
 \square_{\var}^t=-(\nabla_{g_{\cH}(t)}-\frak T^{\var}_{g_{\cH}(t)})^*(\nabla_{g_{\cH}(t)}-\frak T^{\var}_{g_{\cH}(t)})-\frac{1}{\var}\mathbf J^2(t)+\frac{1}{\var}\delta_{\cH}T(t)-\mathfrak{Ric}_{g_{\cH}(t)}.
 \]
 Here we denote $\mathbf J^2(t)$ and $\delta_{\cH}T(t)$ to emphasize their dependence on time. A time independent version of the above operator acting on one forms can be found in \cite{BKW15}.
 
 Now we introduce the \emph{carr\'e du champ} operator \cite{BakryEmery85} associated with the time dependent horizontal Laplacian operator $L^t$.  For $f,g\in \mathcal C^{\infty}(\M)$,  
  \begin{align}\label{Gamma}
  \begin{split}
& \Gamma^t(f,g)=\frac{1}{2}(L^t(fg)-fL^tg-gL^tf)=g_{\cH}(\nabla_{g_{\cH}(t)}f,\nabla_{g_{\cH}(t)}g)=<df,dg>_{g_{\cH}},\\
 &\Gamma^{\cV}(f,g)=g_{\cV}(\nabla_{\cV}f,\nabla_{\cV}g)=<df,dg>_{g_{\cV}}.
 \end{split}
 \end{align}
 Their iteration are defined as 
 \begin{align}\label{Gamma 2}
 \begin{split}
 \Gamma^t_2(f,g)&=\frac{1}{2}(L^t\Gamma^t(f,g)-\Gamma^t(L^tf,g)-\Gamma^t(f,L^tg)),\\
 \Gamma^{\cV}_2(f,g)&=\frac{1}{2}(L^t\Gamma_{\cV}(f,g)-\Gamma_{\cV}(L^tf,g)-\Gamma_{\cV}(f,L^tg)).
 \end{split}
  \end{align}
Both $\square_{\var}^t$ and the \emph{carr\'e du champ} operators will play important role in the Bochner's identity and the generalized curvature dimension inequality.
 
\section{Statement of main results.}
Let $(\M^{n+m},\mathcal F,g_0)$ be a $n+m$-dimensional complete manifold, equipped with bundle like metric $g_0$ \cite{Reinhart59} and totally geodesic foliation structure $\mathcal F$ \cite{Tondeur88,Molino88 ,Baudoin14}. We denote $g_{\var,0}=g_{\cH,0}\oplus \frac{1}{\var}g_{\cV,0}$ as the one parameter family variation of the initial metric $g_0=g_{\cH,0}\oplus g_{\cV,0}$. In general, we assume that the metric evolves in time, we denote the evolution equation as 
\beq \label{general RF}
\begin{cases}
&\pt g_{\cH}(t)=h(t),\quad g_{\var}(0)=g_{\var,0}=g_{\cH,0}\oplus \frac{1}{\var}g_{\cV,0},\\
&\pt g_{\cV}=0.
\end{cases}
\eeq
Assuming that $h(t)$ is a symmetric $(0,2)$ tensor and $h(X,Z)=0$ if $X\in\Gamma^{\infty}(\cV)$ or $Z\in\Gamma^{\infty}(\cV)$, then we can show the bundle like metric and totally geodesic foliation structures are preserved under this flow. Furthermore, if we assume the \emph{Yang-Mills} condition $\delta_{\cH}T(\cdot)=0$ at time $t=0$, the flow will also preserve the \emph{Yang-Mills} condition for any time $t$.\\
We then get a time dependent Bochner-Weitzenb\"ock formula for operator
  \[
 \square_{\var}^t=-(\nabla_{g_{\cH}(t)}-\frak T^{\var}_{g_{\cH}(t)})^*(\nabla_{g_{\cH}(t)}-\frak T^{\var}_{g_{\cH}(t)})-\frac{1}{\var}\mathbf J^2(t)+\frac{1}{\var}\delta_{\cH}T(t)-\mathfrak{Ric}_{g_{\cH}(t)},
 \]
 which is introduced in the previous section.
 \begin{thm}\label{Bochner}
 Let $f\in\mathcal C^{\infty}(\M)$, under the general transverse flow \eqref{general RF}
 \[
 dL^tf=\square_{\var}^tdf,
 \]
 and for $\eta=df$, we have the following time dependent Bochner's inequality
 \begin{align}\label{Bochner's inequality}
 \begin{split}
& \frac{1}{2}L^t\|\eta\|_{\var}^2-<\square_{\var}^t\eta,\eta>_{\var}\\ 
&\geq \frac{1}{n}(\mathbf{Tr}_{\cH}\nabla^{t,\sharp}_{g_{\cH}(t)}\eta)^2-\frac{1}{4}\mathbf{Tr}_{\cH}(\mathbf J^2_{\eta})+<\mathfrak{Ric}_{g_{\cH}(t)}(\eta),\eta>_{\cH}
-<\delta_{\cH}T(\eta),\eta>_{\cV}+\frac{1}{\var}<\bf J^2(\eta),\eta>_{\cH} .
\end{split}
 \end{align}
 \end{thm}
 Then in section $5,6,7$, we consider the following transverse Ricci flow on $(\M^{n+m},\mathcal F, g)$, which has the following form
 \begin{align*}
 \begin{cases}
 & \frac{\partial g_{\cH}}{\partial t}=-2\ricci_{\cH},\quad g_{\var}(0)=g_{\var,0}=g_{\cH,0}\oplus \frac{1}{\var}g_{\cV,0}, \\
 &\frac{\partial g_{\cV}}{\partial t}=0.
 \end{cases}
 \end{align*}
 Namely we take $h(t)=-2\ricci_{\cH}(t)$, where $\ricci_{\cH}(t)$ is the horizontal Ricci curvature for metric $g(t)$. 
 
Under further assumption that for any horizontal one form $\eta_1$, vertical one form $\eta_2$ and constants $\rho_1,\kappa, \rho_2>0$
 \begin{align}\label{assumption 1}
 \begin{split}
 <\mathfrak{Ric}_{\cH}(\eta_1),\eta_1>_{\cH}&\geq \rho_1\|\eta_1\|_{\cH}^2,\\
  -<\bf J^2\eta_1,\eta_1>&\leq \kappa \|\eta_1\|_{\cH}^2,\\
  -\frac{1}{4}\bf{Tr}_{\cH}(\bf J^2_{\eta_2})&\geq \rho_2\|\eta_2\|_{\cV}^2,
  \end{split}
 \end{align}
  and Yang-Mills condition $\delta_{\cH}T(\cdot)=0$. We can prove a time dependent version of the generalized curvature dimension inequality \cite{BKW15}[Theorem 5.1] in the following sense. For details, we refer to section $5$.
 \begin{thm}\label{generalized CD inequality}
 For every $f,g\in \mathcal C^{\infty}(\M)$ and $\var>0,$
 \[
 \Gamma^t_2(f,f)+\var\Gamma^{\cV}_2(f,f)\geq \frac{1}{n}(L^tf)^2+(\rho_1-\frac{\kappa}{\var})\Gamma^t(f,f)+\rho_2\Gamma^{\cV}(f,f),
 \]
 and 
 \[
 \Gamma^t(f,\Gamma^{\cV}(f))=\Gamma^{\cV}(f,\Gamma^t(f)).
 \]
 \end{thm} 
 As an application of Theorem \ref{Bochner} and Theorem \ref{generalized CD inequality}, we can prove the following entropy inequality and gradient estimates using the semigroup representation for the solution of heat equation \eqref{heat equation}.
  \begin{align*}
 (L^t-\pt)u(x,t)=0,\quad x\in\M, t\in[0,T].
 \end{align*}
 We denote $P_{s,T}f$ as the solution at time $T$ with initial value $f$ at time $s.$
 \begin{thm}\label{entropy inequality}Given a function $f\in \mathcal C_b^{\infty}(\M)$ and $\var>0$, we let $f_{\var}=f+\var$.
\begin{align*}
&a(T)P_{s,T}(f_{\var}\Gamma^t(\ln f_{\var}))+b(T)P_{s,T}(f_{\var}\Gamma^{\cV}(\ln f_{\var}))-(P_{s,T}f_{\var})(a(s)\Gamma^t(\ln P_{s,T}f_{\var})+b(s)\Gamma^{\cV}(\ln P_{s,T}f_{\var}))\\
&\geq \int_s^T (a'-2\kappa\frac{a^2}{b}-\frac{4a\gamma}{d})\Phi_1d\tau+\int_s^T(b'+2\rho_2a)\Phi_2d\tau+(\frac{4}{d}\int_s^T a\gamma d\tau)L^tP_{s,T}f_{\var}\\
&-(\frac{2}{d} \int_s^T a\gamma^2d\tau)P_{s,T}f_{\var}.
\end{align*}
\end{thm}
 
\begin{thm}\label{gradient estimate}
Let $b:[0,T]\rightarrow [0,\infty)$ be non-increasing $\mathcal C^2$ function such that, with
\[
\gamma=^{def}\frac{d}{4}(\frac{b''}{b'}+\frac{\kappa}{\rho_2}\frac{b'}{b}), \quad b(t)=(T-t)^3,
\]
we have 
\begin{align*}
\Gamma^t(\ln P_{s,t}f)+\frac{2\rho_2}{3}t\Gamma^{\cV}(\ln P_{s,t}f)&\leq(1+\frac{3\kappa}{2\rho_2} )\frac{L^tP_{s,t}f}{P_{s,t}f}+\frac{d(1+\frac{3\kappa}{2\rho_2})^2}{2(t-s)}.
\end{align*}
\end{thm}
\begin{rem}
Note that, we still assume the Yang-Mills condition $\delta_{\cH}T(\cdot)=0$ in the above theorem. This condition is not necessary, if we further assume \eqref{assumption 2} as we did in the next theorem, the only difference for the above theorem is just $b'+2\rho_2a+2\beta a=0$ in the proof and the computations follow directly. For details, please see the proof in section $4$.
As a consequence, we can get parabolic Harnack inequality and heat kernel upper bound which is similar to the static case as in \cite{BaudoinGarofalo14}.
\end{rem}

In the following, we will forget the semigroup representation of the solution $u(x,t)$ to the heat equation \eqref{heat equation}, we directly work with $u(x,t)$. Under the assumption,
 \begin{align}\label{assumption 2}
 \begin{split}
 -k_1g_{\cH}(x,t)\leq \ricci_{\cH}(x,t)&\leq k_2g_{\cH}(x,t), \\
 -<\bf J^2\eta_1,\eta_1>&\leq \kappa \|\eta_1\|_{\cH}^2,\\
-2<\delta_{\cH}T(\eta_2),\eta_2>_{\cV}&\geq -2\beta\|\eta_2\|_{\cV}^2,\\
-\frac{1}{4}\bf{Tr}_{\cH}(\bf J^2_{\eta_2})&\geq \rho_2\|\eta_2\|_{\cV}^2, 
\end{split}
 \end{align}
 \begin{rem}
 Comparing with assumption \eqref{assumption 1}, we do not assume the Yang-Mills condition $\delta_{\cH}T(\cdot)=0$, instead we have a bound on it and we assume two sided bounds on the horizontal Ricci curvature.
 \end{rem}
 We have another differential Harnack inequality for solution $u(x,t).$
 \begin{thm}\label{diff harnack}
 Suppose $(\M,g(x,t))_{t\in[0,T]}$ is a complete solution to the transverse Ricci flow \eqref{transverse RF}. 
 Given assumption \eqref{assumption 2}. Let $u$ be the positive solution to heat equation \eqref{heat equation} associated with the time dependent horizontal Laplacian operator $L^t$. We have the following differential Harnack inequality: 
 \begin{align}
 \Gamma^t(\log u)+b(t)\Gamma^{\cV}(\log u)-\alpha(t)(\log u)_t\leq \phi(t)+B,
 \end{align}
 where we have 
 \[
 b(t)=\frac{2(\rho_2-\beta)\int_0^td(s)ds}{d(t)},\quad \alpha(t)=\frac{e^{2k_1t}\int_0^te^{-2k_1s}d(s)(\frac{d'(s)}{d(s)}+\frac{\kappa}{\var}+\frac{t_1d_2}{\rho})ds}{d(t)},
 \]
 \[
 \phi(t)=\frac{\alpha(t)n\max\{k_1^2,k_2^2\}\int_0^td(s)ds}{2cd(t)},\quad \eta(t)=-\frac{nk_1}{2a}-\frac{n}{4a\alpha(t)}(\frac{d'(t)}{d(t)}+\frac{\kappa}{\var}+\frac{t_1d_2}{\rho}) ),
 \]
 \[
 B=\frac{1}{4}(\frac{1}{\rho}+F(r_{\var,\gamma})+\frac{\rho}{t}+\rho \bar k-{ \frac{\rho d'(t_1)\Psi(t_1)}{d_2t_1 d(t_1)}}),\quad
 d_2=\max\{3d_1,C_{1/2},3C_{1/2}^2,\bar C\}
 \]
 where $F(r_{\var,\gamma})$ comes from Corollary 2.9 \cite{BGKT17}, certain constant comes from the sub-Laplacian comparison theorem. $d_1$ and $d_2$ are some constants, $t_1$ is the time when the maximum of $\Psi F$ \eqref{Psi function} is attained. $($See the proof of Theorem \ref{diff harnack} in Section 6 for details.$)$
 \end{thm}
 \begin{rem}
 For the above theorem, if we further investigate it on a Sasakian manifold. Then by taking the limit as $\var\rightarrow 0$ for the metric $g_{\var}$, we will actually approximate the sub-Riemannian distance $r_0$ by the Riemannian distance $r_{\var}$. We then get a more specific upper bound where we can replace $F(r_{\var,\gamma})$ by the constant proved in \cite{BGKT17}[Theorem 3.2].
 \end{rem}
\begin{cor}\label{t square}
 Under the assumption of Theorem \ref{diff harnack}, taking $d(t)=t^{2}$, $c=\frac{\alpha(t)}{2}$ , we have for all $t$,
 \begin{align}\label{cor: diff harnack}
 \begin{split}
 \Gamma^t(\log u)+\frac{2(\rho_2-\beta)}{3}\Gamma^{\cV}(\log u)&\leq \alpha(t)(\log u)_t+\frac{n\max\{k_1,k_2\}}{2}\\
 &+\frac{1}{4}(\frac{1}{\rho}+F(r_{\var},\gamma)+\rho \bar k)+\frac{\rho}{4t},
 \end{split}
 \end{align}
 where 
 \[
 \alpha(t)=(2k_1+\frac{\kappa}{\var}+\frac{t_1d_2}{\rho})\frac{e^{2k_1t}}{4k_1^3t^2}+1-(2k_1+\frac{\kappa}{\var}+\frac{t_1d_2}{\rho})(\frac{1}{2k_1}+\frac{1}{2k_1^2t}+\frac{1}{4k_1^3t^2}).
 \]
 \end{cor}
 \begin{rem}
 Of course, different function $d(t)$ will give us different estimates. But $d(t)=t^{2}$, is a suitable function such that we have the coefficient $\frac{2(\rho_2-\beta)}{3}$, which is the same coefficient in our Theorem \ref{gradient estimate} and the static setting result in \cite{BaudoinGarofalo14} without assuming the Yang-Mills condition. Thus we can directly compare the right hand side to see how the estimates is affected under the Ricci flow, the same for the parabolic Harnack inequality in the following.
 \end{rem}

We use the above corollary \ref{cor: diff harnack} to get the parabolic Harnack inequality.
\begin{thm}\label{parabolic harnack}
 Under the assumptions of Theorem \ref{diff harnack} and conditions in Corollary \ref{t square}, suppose $u$ is the solution to the time dependent heat equation, we have the estimates
 \begin{align}
 u(x,s)\leq u(y,t)(\frac{t}{s})^{\frac{\rho}{4}}\exp\left(\frac{T\hat\alpha(k_1,T)}{4(t-s)}r_{}^2(x,y)+A(t-s) \right),
 \end{align}
 where 
 \[
 A=\frac{n\bar k}{2}+\frac{1}{4}(\frac{1}{\rho}+F(r_{\var},\gamma)+\rho\bar k),\quad  \bar k=\max\{k_1,k_2\},
 \]
 \[
 \hat \alpha(k_1,T)=:\begin{cases}
C \frac{e}{2k_1},& \text{if}~ k_1\leq \frac{1}{T^2},\\
C \max\{ \frac{e}{2k_1}, \frac{e^{2k_1T}}{4k_1^3T^2} \}, & \text{if}~ k_1\geq \frac{1}{T^2}.
 \end{cases},\quad
 C=(2k_1+\frac{\kappa}{\var}+\frac{t_1d_2}{\rho}),
 \]
  holds for all $(x,s)\in \M\times(0,T)$ and $(y,t)\in \M\times(0,T)$ such that $s<t$ and $r_{}(x,y)\leq \rho$ where $r_{}$ is the sub-Riemannian (cc) distance.  \end{thm}

 \begin{rem}
 A detailed analysis of function $\alpha(t)$ may give us more concise estimates, where a constant function $\alpha$ version of this parabolic Harnack inequality can be found in \cite{BCP10}[Theorem 2.11]. If we further expand $\alpha(t)$ in a Taylor expansion in higher order terms of $t$, we can get a similar version like \cite{Qian13}[Theorem 4.1]. 
 \end{rem}
A direct consequence of the above theorem \ref{parabolic harnack} is the following heat kernel upper bound.
 \begin{cor}\label{kernel upper bound}
Given the assumption in theorem \ref{diff harnack}. Let $p(x,y,t)$ be the heat kernel for $u(t)$ on $\M$. For every $x,y,z\in\M$ and $0<s<t<T$, we have
 \begin{align*}
 p(x,y,s)\leq u(x,z,t)(\frac{t}{s})^{\frac{\rho}{4}}\exp\left(\frac{T\hat \alpha(k_1,T)}{4(t-s)}r_{}^2(x,y)+A(t-s) \right).
 \end{align*}
\end{cor}
 \section{Structure under transverse Ricci flow}
 We consider the following transverse geometric flow equation on $(\M^{n+m},\mathcal F, g_0)$,
 \begin{align}\label{general transverse RF}
 \begin{cases}
& \pt g_{\cH}(t)=h(t),\quad g(0)=g_0.\\
&\pt g_{\cV}=0,
\end{cases}
 \end{align}
 where $h(t)$ satisfies the following properties: 
 \begin{itemize}
 \item[1]. For any $X\in \Gamma^{\infty}(\cV)$ or $Z\in \Gamma^{\infty}(\cV)$, we have $h(X,Z)=0$
 \item[2]. $h(t)$ is a symmetric $(0,2)$ tensor.
 \end{itemize}
Similar consideration can also be found in \cite{Rovenski13partial}, where the authors studied partial Ricci flow for co-dimension one foliation.
In particular, if we consider the specific transverse Ricci flow for $h(t)=-2\ricci_{\cH}(t)$ where $\ricci_{\cH}(t)$ is the horizontal Ricci curvature for metric $g(t)$.  Locally we have the projection $\Pi: U\in\M\rightarrow \tilde U$, where $\tilde U$ is the local Riemannian quotient. If we denote the Ricci flow on the local Riemannian quotient as
 \beq\label{base RF}
 \frac{\partial g_{\tilde U}}{\partial t}=-2\ricci_{\tilde U}.
 \eeq
 The transverse Ricci flow on $\M$ is 
 \beq \label{transverse RF 1}
 \begin{cases}
 & \frac{\partial g_{\cH}}{\partial t}=-2\ricci_{\cH},\\
 &\frac{\partial g_{\cV}}{\partial t}=0.
 \end{cases}
 \eeq
 Then the pull-back $\Pi^*$ of \eqref{base RF} will give us \eqref{transverse RF 1} directly following from Lemma \ref{bundle like metric}.
 In particular, if the Riemannian foliation comes from a Riemannian submersion, then the Ricci flow $\eqref{base RF}$ is just the Ricci flow on the base manifold $\B$. In this paper, we will always consider Riemannian submersion as a specific example for our result of Riemannian foliations.
 
\begin{example}[\textbf{Riemannian submersions}]\label{ex.RiemSubmersion}

Let $(\M, g)$  and  $(\B,j)$ be smooth and connected Riemannian manifolds. A smooth surjective map $\pi: (\M, g)\to (\B,j)$ is called a \emph{Riemannian submersion} if its derivative maps $T_x\pi : T_x \M \to T_{\pi(x)} \B$ are orthogonal projections, i.e. for every $ x \in \M$, the map $ T_{x} \pi (T_{x} \pi)^*: T_{p(x)}  \B \to T_{p(x)} \B$ is the identity map.
The foliation given by the fibers of a Riemannian submersion is then bundle-like (see  \cite[Section 2.3]{Baudoin14} ).
\end{example}

 \begin{rem}
In general, restrict metric $g$ to $\cH$ gives metric on $\cH$ denoted as $g_{\cH}$, however this metric will not yield a metric on the local Riemannian quotient. In fact, $U$ is a small open set, and $\pi: U\rightarrow \tilde U$ is the quotient map to the local Riemannian quotient, then $g_{\cH}=\pi^*\tilde g$ for some metric $\tilde g$ if and only if $g$ is bundle like. As mentioned in \cite{Collins12}, if we restrict to Riemannian foliation, then the above procedure produces a global metric $g_{\cH}$ on $\cH$, which is given locally by the pull-back of a metric from the local Riemannian quotient.
 \end{rem}

 \begin{thm}\label{thm: flow existence}
The transverse flow \eqref{general transverse RF} and transverse Ricci flow \eqref{transverse RF 1} have unique short-time solution.
 \end{thm}
 \begin{proof}
 The proof simply follows from Theorem $4$ in \cite{Rovenski13partial} and Theorem $5.1$ in \cite{BHV15}. Since our totally geodesic foliation structure satisfies the homologically oriented Riemannian foliation condition and our transverse metric $g_{\cH}(t)$ is a smooth holonomy invariant metric on the quotient bundle. The details we refer to \cite{BHV15}.
 \end{proof}
 \begin{rem}
 The short time existence need to show the weak parabolic property of the linearized operator, since the torsion is just the first order derivative of the metric, it will not affect the short time existence. As for the long time existence, we will need the evolution equation of the torsion $T$ for the $Bott$ connection $\nabla^t$ which will appear in another paper.
 If we only consider basic functions, then the $Bott$ connection acting on basic functions is the same as the transverse Levi-Civita connection,which means we also have torsion free in this basic function case. 
  \end{rem}
  \begin{rem} 
 In the following, we want to prove that under our transverse flow \eqref{general transverse RF}, the bundle like metric and totally geodesic foliation structure are preserved. And the transverse Ricci flow \eqref{transverse RF 1} is a special case of \eqref{general transverse RF}, it will also preserve these two properties.
 \end{rem}
 Recall that a Riemannian metric $g$ is said to be bundle like \cite{Reinhart59} with respect to the foliation induced by the foliation $\cF$ if for any open set $U$ and foliate vector fields $X,Y$ on U perpendicular to $\cV$, the function $g(X,Y)$ is basic. Namely $g(X,Y)$ vanishes on vertical bundle $\cV$.
 
 It is a easy fact to check that $\nabla^t$ is a metric connection.
 \begin{lem}
  The time-dependent Bott connection is a metric connection,which means for any vector fields $X,Y,Z\in T\M$
 \beq\label{metric connection}
 \nabla^t_Xg(Y,Z)=g(\nabla^t_XY,Z)+g(\nabla^t_XZ,Y)
 \eeq
 \end{lem}
 \begin{lem}
 \label{bundle like metric}
 Under the transverse Ricci flow \eqref{transverse RF 1} $($transverse flow \eqref{general transverse RF}$)$, Riemannian foliation with bundle like metric is preserved. 
 \end{lem}
 \begin{rem}
 In the following, we will show in the general case for the transverse flow \eqref{general transverse RF}. Then for the transverse Ricci flow \eqref{transverse RF 1}, the properties follow directly, since $h(t)=-2\ricci_{\cH}$ satisfies our assumptions on $h(t)$.
 \end{rem}
 \begin{rem}
Before we prove Lemma \ref{bundle like metric}, let's first show that the  transverse flow \eqref{general transverse RF} preserves the Riemannian submersion structure, so we have local orthonormal frames for Riemannian foliation.
 \end{rem}
 \begin{lem}\label{Riemannian submersion structure}
If the Riemannian foliation comes from a Riemannian submersion, $\Pi: \M \rightarrow  \B$ with $\mathcal F$ as the totally geodesic foliation, then the transverse flow \eqref{general transverse RF} preserves the Riemannian submersion structure.
 \end{lem}
 
 \begin{proof}[Proof: Riemannian submersion case]
 The Riemannian submersion structure is preserved under the transverse (Ricci \eqref{transverse RF 1}) flow \eqref{general transverse RF} which means that 
 $\pi: (\M,g_t)\rightarrow (\mathbb B,j_t)$ is a Riemannian submersion for any time $t\geq 0.$ \
Let's fix a point $x\in\M$ first, at time $t=0$, we know that $\pi: (\M,g)\rightarrow (\mathbb B,j)$ is a Riemannian submersion, so for any vector fields $\bar X\in \Gamma^{\infty}(T_{\pi(x)}\B)$, there is a unique basic vector fields $X\in \Gamma^{\infty}(T_x\M)$ which is horizontal and $\pi$-related to $\bar X$ such that $(T_x\pi) X(x)=\bar X(\pi(x))$. 
 Now it is left to show that for any  $\bar X(t) \in \Gamma^{\infty}(T_{\pi(x)}\B)$, there is a unique basic vector field $X(t)\in \Gamma^{\infty}(T_x\M)$ such that $(T_x\pi) X(t)=\bar X(t).$ Then we will have $(T_x\pi)(T_x\pi)^*\bar X(t,x)=\bar X(t,\pi(x))$ for any $t.$ 
From the transverse (Ricci) flow, we know that $ \frac{\partial g_{\cV}}{\partial t}=0$, so $g_{\cV}(t)=g_{\cV}$, namely
the vertical metric is constant, which means the vertical distribution is fixed and the vertical orthonormal frame $\{Z_1,\cdots, Z_m\}$ does not depend on time (we then pick one basis at $t=0$). For $x\in \M$, the orthogonal complement of $\cV_x$ is $\cH_x$, 
 so the horizontal distribution is also fixed, but the horizontal orthonormal frames will depend on $t$, since $g_{\cH}(t)$ evolves in time. 
 It is also easy to verify that for any $X\in \Gamma^{\infty}(\cH)$ and $Z\in \Gamma^{\infty}(\cV)$, we have 
 \[
 g_t(X,Z)=0
 \]

 Thus we can pick a orthonormal horizontal basis $\{\bar X_1,\cdots, \bar X_n\}$ at time $t=0$ for $T_{\pi(x)}\B$, then there is a unique basic vector field basis $\{X_1,\cdots, X_n\}$ which are horizontal and projectable, $T_x\pi X_i(x)=\bar X_i(\pi(x)).$ Then at time $t$, any vector fields $\bar X(t)\in \Gamma^{\infty}(T_x\B)$ can be represented in terms of the basis at time $t=0,$ denoted as $\bar X(t)=\sum_{i=1}^na_i(t)\bar X_i$, then it is obvious that there exists $X(t)=\sum_{i=1}^na_i(t)X_i$ is horizontal and $\pi$-related to $\bar X(t)$ 
 \[
 (T_x\pi) X(t)=(T_x\pi) \sum_{i=1}^na_i(t) X_i=\sum_{i=1}^na_i(t) (T_x\pi) X_i=\sum_{i=1}^na_i(t)\bar X_i=\bar X(t)
 \]
 this is also true for $t=0$ with $a_i(0)$ as initial condition.\\
 If basic vector fields $X(t)$ is not unique, suppose there is another basic vector fields $\tilde X(t)=\sum_{i=1}^nb_i(t)X_i$ for $\bar X(t)$. Since we know at $t=0$, it is unique, so $a_i(0)=b_i(0)$ for $i=1,\cdots,n.$ For time $t$, 
 \[
 (T_x\pi)\tilde X(t)=\bar X(t)\quad \Rightarrow (T_x\pi)\sum_{i=1}^nb_i(t)X_i=\sum_{i=1}^na_i(t)\bar X_i
 \]
comparing the coefficients of the basis $\bar X_i$, we get $a_i(t)=b_i(t)$ for any $t\geq 0$, so basic vector fields $X(t)$ is unique.
  \begin{align*}
 j_t(\bar X(t), (T_x\pi)(T_x\pi)^*\bar X(t))&=j_t((T_x\pi)X(t),(T_x\pi)(T_x\pi)^*\bar X(t) )\\
 &=((T_x\pi)^*j_t)(X(t), (T_x\pi)^*\bar X(t))\\
&=g_t(X(t),(T_x\pi)^*\bar X(t))\\
 &=g_t(X(t), X(t))\\
 &=j_t((T_x\pi)X(t),(T_x\pi)\tilde X(t))=j_t(\bar X(t),\bar X(t))
 \end{align*}
 Note that we have $((T_x\pi)^*j_t)=g_t$ not $g_{\cH}(t)$, since we do not have bundle like condition at time $t$ yet. $(T_x\pi)^*\bar X(t)$ is not necessary a horizontal vector field. But since $X(t)$ is horizontal, $g_t(X(t),((T_x\pi)^*\bar X(t))_{\cV})=0$, the above argument follows. We thus get for any time $t$, for arbitrary $\bar X(t)\in \Gamma^{\infty}(T_x\B)$
 \[
 (T_x\pi)(T_x\pi)^*\bar X(t)=\bar X(t)
 \]
 so $\pi: (\M,g_t)\rightarrow (\mathbb B,j_t)$ is a Riemannina submersion for any time $t\geq 0.$ As a consequence, $g(t)$ is a bundle like metric for any $t\geq 0.$
  \end{proof}
  \begin{rem}
In another word, if $\{Z_1,\cdots,Z_m\}$ is orthonormal at time $t=0$, then they are also orthonormal for all $t>0.$ the transverse (Ricci) flow preserve the length of vertical orthonormal vector fields. This is similar to co-dimension $1$. This flow also preserve the vector fields orthogonal to vertical distribution. Namely, 
 \[
 g_t(X(t),Z)=0 \quad 
 \]
 this is true for $t=0$ and taking derivatives gives us the desired result.
 \end{rem}
 \begin{lem}\label{ON frame}
 Let $x\in \M$. Around $x$, at any time $t$, there exists a local orthonormal horizontal frame $\{X_1(t),\cdots,X_n(t) \}$ and a local orthonormal vertical frame $\{Z_1,\cdots,Z_m\}$ such that the following structure relations hold
 \[
[X_i(t),X_j(t)]=\sum_{k=1}^n \omega_{ij}^k(t) X_k(t) +\sum_{k=1}^m \gamma_{ij}^k(t) Z_k
\]
\[
[X_i(t),Z_k]=\sum_{j=1}^m \beta_{ik}^j(t) Z_j,
\]
where $\omega_{ij}^k(t),  \gamma_{ij}^k(t),  \beta_{ik}^j(t) $ are smooth functions in space and time such that:
\[
 \beta_{ik}^j(t)=- \beta_{ij}^k(t).
\]
Moreover, at $x$, we have
\[
 \omega_{ij}^k(t)=0,  \beta_{ij}^k(t)=0.
\]
 \end{lem}
 \begin{proof}
 The proof almost follows from the static case as in \cite{BKW15}. We fix $x$ through the proof, locally the transverse (Ricci) flow preserve Riemannian submersion property, thus we have bundle like metric. Let $X_1(t), \cdots, X_n(t)$ be a local orthonormal horizontal frame around x consisting of basic vector fields for the submersion at time $t$, this is true for $t=0$ (see \cite{BKW15}) and thus for any fixed $t$. Since the metric $g(t)$ is smooth in $t$, so is the connection and the vector fields $X_i(t)$. Let $Z_1,\cdots, Z_m$ be any local orthonormal vertical frame around x, and this will work for any time $t$ since the flow preserve the vertical distribution. Thus the rest of proof follows the same as Lemma $2.2$ in \cite{BKW15}.
 \end{proof}
 We record the fact that in this frame the Christofel symbols of the Bott connection $\nabla^t$ are given by
\begin{align*}
\begin{cases}
\nabla^t_{X_i(t)} X_j(t) =\frac{1}{2} \sum_{k=1}^n \left( \omega_{ij}^k(t) +\omega_{ki}^j(t)+\omega_{kj}^i(t)\right)X_k(t) \\
\nabla^t_{Z_j} X_i(t) =0 \\
\nabla^t_{X_i(t)} Z_j=\sum_{k=1}^m \beta_{ij}^{k}(t) Z_k
\end{cases}
\end{align*}
  With above property in hand, we are ready to prove Lemma \ref{bundle like metric}.
  \begin{proof}[Proof for Lemma \ref{bundle like metric}: Riemannian foliation with bundle like metric] Under the transverse (Ricci \eqref{transverse RF 1}) flow \eqref{general transverse RF}, we just need to show that on any open set $U$ and foliate vector fields $X,Y\in U$ with $X,Y\in \cH$, the function $g_t(X,Y)$ is basic. We need to check that $\cL_Zg_t(X,Y)=0$, for any $Z\in\cV$. Here $Z$ does not depend on time, since under transverse (Ricci) flow \eqref{general transverse RF}, $\pt g_{\cV}=0$, the vertical distribution is fixed. According to Lemma \ref{ON frame},
  around $x\in U$, at any time $t$, there exists a local orthonormal horizontal frame $\{X_1(t),\cdots,X_n(t) \}$ and a local orthonormal vertical frame $\{Z_1,\cdots,Z_m\}$. Thus we only need to work with orthonormal frames. We know that 
$\cL_Zg_t(X,Y)=0$ at $t=0$, so we just need to check $\pt (\cL_Zg_t(X,Y))=0.$ $
( \cL_{Z}g_t)(X,Y)=-g(Z,\nabla^t_XY+\nabla^t_YX)$, thus
 \begin{align*}
 \pt (\cL_Zg_t(X,Y))&=-h(Z,\nabla^t_XY+\nabla^t_YX)-g(Z,\pt(\nabla^t_XY+\nabla^t_YX))\\
 &-g(\pt Z,\nabla^t_XY+\nabla^t_YX)
 \end{align*}
 since $\nabla^t_XY+\nabla^t_YX$ is horizontal and $Z$ does not depend on time, so the first and third terms are zeros. The second term is also zero since
 \[
 g(Z,\pt(\nabla^t_XY))=g(Z,\pt\nabla^t_XY)+g(Z,\nabla^t_{\pt X}Y)+g(Z,\nabla^t_X(\pt Y))
 \]
The first term of $R.H.S$ is similar for torsion free case as in \cite{Topping06}. We can easily compute and have the following form
   \begin{align*}
 < \pt \nabla^t_XY,Z>&=\frac{1}{2} \{(\nabla^t_Xh)(Y,Z)-(\nabla^t_Zh)(X,Y)+(\nabla^t_Yh)(Z,X) \}\\
  &+\frac{1}{2}\{ <X,\frac{\partial}{\partial t}([Z,Y]+T(Z,Y))> - <Y,\frac{\partial}{\partial t}([X,Z]+T(X,Z))>\\
& -<Z,\frac{\partial}{\partial t}([Y,X]+T(Y,X))>\}
 \end{align*}
  simply taking $X=X_i(t),Y=X_j(t)$ for $\forall i,j\in \{1,\cdots,n\}$ and $Z=Z_k$ for $\forall~ k\in \{1,\cdots,m\}$.  Since $[X_i,X_j]+T(X_i,X_j)\in \Gamma^{\infty}(\cH)$, $[X_i,Z_k]\in \Gamma^{\infty}(\cV)$, $\nabla_{Z_k}X_i=0$, $\nabla_{X_i}Z_k\in\Gamma^{\infty}(\cV)$, all the terms vanish in the above equation. The second and third terms of $R.H.S$ are zeros since $\pt X_i(t)\in \cH$. The same reason for $< \pt \nabla^t_YX,Z>=0$. So we have proved that $\cL_Zg_t(X,Y)=0$, thus the bundle like metric are preserved for Riemannian foliation evolve under transverse (Ricci \eqref{transverse RF 1}) flow \eqref{general transverse RF}. It is necessary to point it out that locally we always have Riemannian submersion, so the metric $g(t)$ always has a orthogonal decomposition locally, so g(X,Z)=0 for $Z\in \cV$ and $X\in \cH.$
  \end{proof}
 
 Recall that in the static case without time involved, we have
 \begin{prop}[\cite{Baudoin14}]
 The submersion $\pi$ has totally geodesic fibers if and only if the flow
generated by any basic vector field induces an isometry between the fibers.
 \end{prop}
 In order to show the above proposition, we just need to show that for any horizontal vector field $X$ and for any vertical vector field $Z_1,Z_2$
 \[
(\cL_Xg)(Z_1,Z_2)=-2g(X,\nabla^R_{Z_1}Z_2)=0 \quad\text{iff}\quad \nabla^R_{Z_1}Z_2\quad \text{is always vertical}
 \]

  \begin{lem}
 \label{totally geodesic condition}
 Given transverse $($Ricci \eqref{transverse RF 1}$)$ flow \eqref{general transverse RF}, it will preserve the totally geodesic foliation condition. which means for any time-dependent vector fields $X_i(t)\in \cH$ and any $Z_j\in\cV$, any basic time-dependent vector fields induces an isometry between fibers.
  \end{lem}
 \begin{proof}
 We follow the previous strategy, 
 \[
( \cL_{X(t)}g)(Z_1,Z_2)=g(\nabla^t_{Z_1}X(t),Z_2)+g(\nabla^t_{Z_2}X(t),Z_1)
 \]
 since $\nabla^t$ is metric connection and $X(t)$ is orthogonal to $Z_2$, so we have $g(\nabla^t_{Z_1}X(t),Z_2)=-g(X(t),\nabla^t_{Z_1}Z_2)$, similarly we have $g(\nabla^t_{Z_2}X(t),Z_1)=-g(X(t),\nabla^t_{Z_2}Z_1)$, so we also have 
 \[
 ( \cL_{X(t)}g)(Z_1,Z_2)=-2g(X(t),\nabla^t_{Z_1}Z_2)
 \]
 at time $t=0$, the totally geodesic foliation condition gives us zero immediately. We only left to check that $ \pt g(X(t),\nabla^t_{Z_1}Z_2)=0$, which means we need to compute 
 \[
 -2h(X(t),\nabla^t_{Z_1}Z_2)+g(\pt X(t), \nabla^t_{Z_1}Z_2)+g(X(t),\pt(\nabla^t_{Z_1}Z_2))
 \]
 at time $t=0$, the foliation is totally geodesic, so that $\nabla^t_{Z_1}Z_2$ is vertical, and since $g_{\cV}$ is fixed so that $\nabla^t$ on vertical distribution is also fixed, does not depend on time, so $\nabla^t_{Z_1}Z_2$ is vertical, thus the first term is zero. Since $\pt X(t)$ is horizontal, the second term is thus zero. The last term is zero since 
 $\nabla^t_{Z_1}Z_2=\nabla_{Z_1}Z_2$ is time independent.
 \end{proof}
 In the end, we will show that
 \begin{lem} If $(\M,\mathcal F,g_0)$ satisfies the Yang-Mills condition $\delta_{\cH}T=0$ at time $t=0$, then the transverse $($Ricci \eqref{transverse RF 1}$)$ flow \eqref{general transverse RF} preserves the Yang-Mills condition.
 \end{lem}
 \begin{proof}
 By the definition of $\delta_{\cH}T(\cdot)$ and $\ricci(\cdot , \cdot)$, it is easy to check that $\delta_{\cH}T(\cdot)=0$ is equivalent to $\ricci(X,Z)=0$ if $X\in \cH, Z\in\cV$. Since $\ricci(X,Z)=0$ for $X\in \cH, Z\in\cV$ for any time $t$, so we are done.
 \end{proof}

 \section{Time dependent generalized curvature dimension inequality and gradient estimates}
 \subsection{Time dependent generalized curvature dimension inequality}
  In this part, we start to look at a one parameter family of the variation of the metric $g_{\var}(t)=g_{\cH}(t)\oplus \frac{1}{\var}g_{\cV}.$
 According to Lemma \ref{ON frame}, we know that our orthonormal frames depend on time $t$, so the Weitzenbock identity in \cite{BKW15} will be time dependent in our situation. But since we know that the transverse (Ricci \eqref{transverse RF 1}) flow \eqref{general transverse RF} preserves the totally geodesic foliation structure and bundle like metric, so at each time $t$ we always have the Weitzenbock identity. Namely for metric $g_{\var}(t)$ and $Bott$ connection $\nabla^t$, we have the following time dependent horizontal laplacian operator
 \[
 L^t=-\nabla^*_{g_{\cH}(t)}\nabla_{g_{\cH}(t)}(\text{or}=-\nabla_{\cH}^*\nabla_{\cH})
 \]

and the following operator $\square_{\var}^t$
 \[
 \square_{\var}^t=-(\nabla_{g_{\cH}(t)}-\frak T^{\var}_{g_{\cH}(t)})^*(\nabla_{g_{\cH}(t)}-\frak T^{\var}_{g_{\cH}(t)})-\frac{1}{\var}\mathbf J^2(t)+\frac{1}{\var}\delta_{\cH}T(t)-\mathfrak{Ric}_{g_{\cH}(t)}
 \]
 we also have the \emph{car\'e du champ} operator as we defined in section $2$. If we take $f=g$ in our definition \eqref{Gamma}, \eqref{Gamma 2}, then for $f\in \mathcal C^{\infty}(\M)$, we use $\Gamma^t(f)=\Gamma^t(f,f)$ for simplicity, then we have
 \[
 \Gamma^t(f)=\|\nabla_{g_{\cH}(t)}f\|_{\cH}=\|df\|_{\cH},\quad \Gamma^{\cV}(f)=\|\nabla_{\cV}f\|_{\cV}=\|df\|_{\cV}
 \]
 and their iteration are given by 
 \begin{align*}
 \Gamma^t_2(f)&=\frac{1}{2}(L^t\Gamma^t(f)-2\Gamma^t(L^tf,f)),\quad  \Gamma^{\cV}_2(f)&=\frac{1}{2}(L^t\Gamma_{\cV}(f)-2\Gamma_{\cV}(L^tf,f))
  \end{align*}
  With all the ingredients in hand, we are ready to prove Theorem \ref{Bochner} and Theorem \ref{generalized CD inequality}.
 \begin{proof}{Proof of Theorem \ref{Bochner} and Theorem \ref{generalized CD inequality}.}
 The proof of these two theorems follows the same computations as in \cite{BKW15}[Theorem 3.1, Theorem 5.1] by using our time dependent orthonormal frames \ref{ON frame}. In particular, we prove Theorem \ref{Bochner} under general transverse flow \eqref{general transverse RF}, but for Theorem \ref{generalized CD inequality}, we only prove it under the transverse Ricci flow \eqref{transverse RF 1} for simplicity. Although, one can also get a version of Theorem \ref{generalized CD inequality} under general transverse flow \eqref{general transverse RF}.
 \end{proof}
\subsection{Gradient estimates: entropy inequality}

In this section, we first introduce the solution to the heat equation of time dependent laplacian $\Delta_{g_{\cH}(t)}$, namely $L^t$ as we introduce in section 2. We denote $P_{s,T}f$ for the solution at time $T$ with initial condition $f$ at time $s$. Similar representation is also used in \cite{HaslhoferNaber15weak,LiLi17} and we refer more discussion on this solution therein. In fact, we have 
\[
\partial_t P_{s,t}f=L^tP_{s,t}f,\quad \partial_s P_{s,t}f=-P_{s,t}L^sf
\]
we can get the following $L^{\infty}$ global parabolic comparison theorem which is a time dependent version of \cite{BaudoinGarofalo14}[Proposition 4.5].
\begin{prop}\label{time parabolic com}
Suppose $\M$ is stochastic complete with $P_t1=1$. Let $u,v:\M\times [0,T]\rightarrow \mathbb R$ be smooth functions such that for every $T>0$, $\sup_{t\in\{0,T\}}\|u(\cdot,t)\|_{\infty}<\infty$, $\sup_{t\in\{0,T\}}\|v(\cdot,t)\|_{\infty}<\infty$; If the inequality 
\[
L^tu+\pt u\geq v,
\]
holds on $\M\times[0,T]$, then we have 
\[
P_{s,T}(u(\cdot,T))(x)\geq u(x,s)+\int_s^TP_{\tau,s}(v(\cdot,\tau))(x)d\tau
\]
\end{prop}
\begin{proof}
At each time $t$, the transverse Ricci flow preserves the totally geodesic foliation structure, so $\M$ is always stochastic complete, we can construct $(X_t^x)$ be the diffusion Markov process with semigroup $P_{s,t}$ and started at $(x,s)$. The rest of the proof follows the same as \cite{BaudoinGarofalo14}, except that we use $\mathbb E(\cdot|\mathcal F_s)$.
\end{proof}
We now introduce the following two functions. For $f\in \mathcal C_b^{\infty}(\M)$, we denote
\[
\Phi_1(x,t)=(P_{s,T-t}f)(x)\Gamma^t(\ln P_{s,T-t}f)(x)
\]
\[
\Phi_2(x,t)=(P_{s,T-t}f)(x)\Gamma^{\cV}(\ln P_{s,T-t}f)(x)
\]
By direct computations similar to \cite{BaudoinGarofalo14}[Lemma 5.1], we have
\begin{lem}
\begin{align*}
L^t\Phi_1+\pt\Phi_1&=2P_{s,T-t}\Gamma_2(\ln P_{s,T-t}f)+P_{s,T-t}(\pt g_{\cH})(\nabla_{\cH}\ln P_{s,T-t}f, \nabla_{\cH}\ln P_{s,T-t}f)\\
L^t\Phi_2+\pt\Phi_2&=2(P_{s,T-t}f)\Gamma^{\cV}_2(\ln P_{s,T-t} f)
\end{align*}
\end{lem}

Now we are ready to prove the entropy inequality  Theorem \ref{entropy inequality}.
\begin{proof}{Proof of theorem \ref{entropy inequality}}. Denote $\Phi(t)$ as 
\[
\Phi(t)=a(t)\Phi_1(t)+b(t)\Phi_2(t)
\]
According to the previous lemma, we have for $f\in \mathcal C_b^{\infty}(\M)$ and $\var>0$, $f_{\var}=f+\var$, use $f_{\var}$ instead of $f$ in the previous lemma and take limit afterwards which is similar to the static case \cite{BaudoinGarofalo14}
\begin{align*}
L^t\Phi+\pt\Phi&=a'(P_{s,T-t})\Gamma^t(\ln P_{s,T-t}f_{\var})+b'(P_{s,T-t}f_{\var})\Gamma^{\cV}(\ln P_{s,T-t}f_{\var})+\\
&+2aP_{s,T-t}f_{\var}\Gamma_2(\ln P_{s,T-t}f_{\var})+aP_{s,T-t}f_{\var}(\pt g_{\cH})(\nabla_{\cH}\ln P_{s,T-t}f_{\var}, \nabla_{\cH}\ln P_{s,T-t}f_{\var})\\
&+2b(P_{s,T-t}f_{\var})\Gamma^{\cV}_2(\ln P_{s,T-t}f_{\var})
\end{align*}
according to Bochner's inequality \eqref{Bochner's inequality} and transverse Ricci flow $\pt g_{\cH}=-2\mathfrak{Ric}_{\cH}$
\begin{align*}
&2aP_{s,T-t}\Gamma_2(\ln P_{s,T-t}f_{\var})+2b(P_{s,T-t}f_{\var})\Gamma^{\cV}_2(\ln P_{s,T-t}f_{\var})\\
&+aP_{s,T-t}f_{\var}(\pt g_{\cH})(\nabla_{\cH}\ln P_{s,T-t}f_{\var}, \nabla_{\cH}\ln P_{s,T-t}f_{\var})\\
&\geq 2aP_{s,T-t}f(\frac{1}{d}(L^t\ln P_{s,T-t}f_{\var})^2+(-\kappa\frac{a}{b})\Gamma^t(\ln P_{s,T-s}f_{\var})+\rho_2\Gamma^{\cV}(\ln P_{s,T-t}f_{\var}))
\end{align*}
Plug the above inequality into $L^t\Phi+\pt\Phi$, we thus have
\begin{align*}
L^t\Phi+\pt\Phi &\geq(a'-2\kappa\frac{a^2}{b}-\frac{4a\gamma}{d})(P_{s,T-t}f_{\var})\Gamma^t(\ln P_{s,T-t}f_{\var})\\
&+(b'+2\rho_2a)(P_{s,T-t}f_{\var})\Gamma^{\cV}(\ln P_{s,T-t}f_{\var})\\
&+\frac{4a\gamma}{d}L^tP_{s,T-t}f_{\var}-\frac{2a\gamma^2}{d}P_{s,T-t}f_{\var}
\end{align*}
now we can apply proposition \ref{time parabolic com} to finish the proof.
\end{proof}
With the above inequality in hand, we are now ready to prove the following Li-Yau type gradient estimates Theorem \ref{gradient estimate}.

\begin{proof}{Proof of theorem \ref{gradient estimate}}.
First choose $\mathcal C^1$ function $a$, such that $b'+2\rho_2a=0$, and with our choice of function of $b$ and $\gamma$, plugging into theorem \eqref{entropy inequality}, we get the desired result.
\end{proof}

 \section{Li-Yau gradient estimates for positive solution of heat equation }

In this section, instead of using the semigroup representation $P_{s,t}f$ for the solution of the heat equation, we directly denote $u(x,t)$ as the solution to the heat equation at time $t$ associated with the time dependent horizontal Laplacian, namely we consider
\[
(\pt-L^t)u(x,t)=0,\quad x\in\M, t\in[0,T].
\]
In this setting, it is more clear and convenient to think that the metric evolves in time and the heat spreads over the manifold at the same time. For more details, we refer to \cite{BCP10}. We are going to get a time dependent version of gradient estimates for $u(x,t).$ We will use similar methods as in \cite{Zhang06, BCP10}. So we recall the following lemma first.
 \begin{lem}[lemma 2.1 \cite{BCP10}.]\label{Lemma Psi}
 Given $\tau\in (0,T]$, there exists a smooth function $\bar{ \Psi}:[0,\infty)\times {0,T}\rightarrow \R$ satisfying the following requirements:
 \begin{itemize}
 \item 1. The support of $\bar{\psi}(r,t)$ is a subset of $[0,\rho]\times[0,T]$, and $0\leq \bar\Psi(r,t)\leq 1$ in $[0,\rho]\times[0,T]$.
 \item 2. The equalities $\bar\Psi(r,t)=1$ and $\frac{\partial \bar\Psi}{\partial r}=0$ hold in $[0,\frac{\rho}{2}]\times[\tau,T]$ and $[0,\frac{\rho}{2}]\times[0,T]$ respectively.
 \item 3. The estimate $|\frac{\partial \bar\Psi}{\partial t}|\leq \frac{\bar C\bar\Psi^{\frac{1}{2}}}{\tau}$ is satisfied on $[0,\infty)\times [0,T]$ for some constants $\bar C>0$, and $\bar \Psi(r,0)=0$ for all $r\in [0,\infty).$
 \item 4. The inequalities $-\frac{C_l\bar\Psi^{l}}{\rho}\leq \frac{\partial \bar\Psi}{\partial r}\leq 0$ and $|\frac{\partial^2 \bar\Psi}{\partial r^2}|\leq \frac{C_l\bar\Psi^l}{\rho^2}$ hold on $[0,\infty)\times [0,T]$ for every $l\in(0,1)$ with some constant $C_l$ dependent on $l$.
 \end{itemize}
 \end{lem}
 Before we prove our main theorem, we prove the following lemma which will play an important role later, this can be viewed as an analogue of \cite{BCP10}[Lemma 2.6].
 \begin{lem}\label{Lemma estimate F}
 Suppose $(\M,g(x,t))_{t\in[0,T]}$ is a complete solution to the transverse Ricci flow \eqref{transverse RF}. Given assumption \eqref{assumption 2}, for any horizontal one form $\eta_1$ and vertical one form $\eta_2$. Under our assumption \eqref{assumption 2}
 \begin{align}
 &-k_1g_{\cH}(x,t)\leq \ricci_{\cH}(x,t)\leq k_2g_{\cH}(x,t), \quad -<\bf J^2\eta_1,\eta_1>\leq \kappa \|\eta_1\|_{\cH}^2\nonumber\\ \nonumber
&-2<\delta_{\cH}T(\eta_2),\eta_2>_{\cV}\geq -2\beta\|\eta_2\|_{\cV}^2,\quad -\frac{1}{4}\bf{Tr}_{\cH}(\bf J^2_{\eta_2})\geq \rho_2\|\eta_2\|_{\cV}^2 \nonumber
 \end{align}
 for some $k_1, k_2, \rho_2,\beta,\kappa>0$. Suppose that $u:\M\times [0,T]\rightarrow \R$ is a smooth positive function satisfying the heat equation \eqref{heat equation}
 \[
 (L^t-\pt)u(x,t)=0,\quad x\in\M, t\in[0,T].
 \]
 For some positive function $b(t)$ and some function $\alpha(t),\phi(t)$, define $f=\log u$ and 
 \begin{align}
 F=\Gamma^t(f)+b(t)\Gamma^{\cV}(f)-\alpha(t)f_t-\phi(t).
 \end{align}
 the estimate
 \begin{align*}
 (L^t-\pt)F+2\Gamma^t(f,F)&\geq(-2\alpha(t)k_1-\frac{\kappa}{\var}-\frac{4a\alpha(t)\eta(t)}{n})\Gamma^t(f)\\
 &+(2\rho_2-2\beta-b'(t))\Gamma^{\cV}(f)+(\alpha'(t)+\frac{4a\alpha(t)\eta(t)}{n})f_t\\
 &+\phi'(t)-\frac{\alpha(t)n}{2c}\max\{k_1^2,k_2^2\}
 \end{align*} 
 holds for any $a,c>0$ such that $a+c=\frac{1}{\alpha(t)}$.
 \end{lem}
 \begin{rem}
 The result is also true for super transverse Ricci flow 
 \begin{align}\label{super transverse RF}
 \pt g_{\cH}\geq -2\ricci_{\cH}
 \end{align}
 the proof is almost the same, we skip this part for simplicity and just give necessary comments in the proof.
 \end{rem}
 \begin{proof}
 By direct computations, we have 
 \begin{align*}
 (L^t-\pt)F&=L^t\Gamma^t(f)+b(t)L^t\Gamma^{\cV}(f)-\alpha(t)L^tf_t\\
 &-2\Gamma^t(f,f_t)-2b(t)\Gamma^{\cV}(f,f_t)-b'(t)\Gamma^{\cV}(f)\\
 &+\alpha(t)'f_t+\alpha(t)f_{tt}+\pt g_{\cH}(\nabla f,\nabla f)-\phi'(t)
 \end{align*}
 due to the fact $L^tf=f_t-\Gamma^t(f)$, taking derivative on both sides gives
 \begin{align*}
L^t f_t+2R_{ij}f_{ij}=f_{tt}-2\Gamma^t(f,f_t)+\pt g_{\cH}(\nabla f,\nabla f)
 \end{align*}
Notice that super transverse Ricci flow \eqref{super transverse RF} gives the above relation $"\leq"$ instead of $"="$.

 Combining with relation \eqref{Gamma 2}, assumption \eqref{transverse RF} (or \eqref{super transverse RF}) and Bochner's inequality \eqref{Bochner's inequality} we further get 
\begin{align*}
 (L^t-\pt)F&=2\Gamma^t_2(f)+2b(t)\Gamma^{\cV}_2(f)-2\Gamma^t(f,F)+\pt g_{\cH}( df, df)\\
 &+\alpha(t)(2R_{ij}f_{ij}-\pt g_{\cH}(\nabla f,\nabla f))-b'(t)\Gamma^{\cV}(f)+\alpha'(t)f_t+\phi'(t)\\
 &=-2\Gamma^t(f,F)+2\|\nabla_{\cH}df-\mathfrak T^{\var}_{\cH}df\|_{\var}-2<\delta_{\cH}T(df),df>_{\cV}+\frac{1}{\var}<\mathbf J^2(df),df>_{\cH}\\
 &+\alpha(t)(2R_{ij}f_{ij}-\pt g_{\cH}(\nabla f,\nabla f))-b'(t)\Gamma^{\cV}(f)+\alpha'(t)f_t+\phi'(t)\\
& \geq -2\Gamma^t(f,F)+ 2\|\nabla_{\cH}^{\sharp}df\|_{\cH}-\frac{1}{2}\mathbf{Tr}_{\cH}(\mathbf J^2_{df})-2<\delta_{\cH}T(df),df>_{\cV}+\frac{1}{\var}<\mathbf J^2(df),df>_{\cH}\\
 &+\alpha(t)(2R_{ij}f_{ij}-\pt g_{\cH}(\nabla f,\nabla f))-b'(t)\Gamma^{\cV}(f)+\alpha'(t)f_t+\phi'(t)
 \end{align*} 
 Our next step is to bound the R.H.S of the last inequality, choose $a,c$ such that $a+c=\frac{1}{\alpha(t)}$. Using the local coordinate representation (e.g. Prop $3.6$, \cite{BKW15}),we have 
 \begin{align*}
 \|\nabla_{\cH}^{\sharp}df\|_{\cH}+\alpha(t)R_{ij}f_{ij}&=\sum_{i,j=1}^n((a\alpha(t)+c\alpha(t))f_{ij}^2+\alpha(t)R_{ij}f_{ij})\\
 &=\sum_{i,j=1}^n(a\alpha(t)f_{ij}^2+\alpha(t)(\sqrt{c}f_{ij}+\frac{R_{ij}}{2\sqrt{c}})^2-\frac{\alpha(t)}{4c}R_{ij}^2)\\
& \geq \frac{a\alpha(t)}{n}(L^tf)^2-\frac{\alpha(t)n}{4}\max\{k_1^2,k_2^2\}
 \end{align*}
 choose function $\eta(t)$ such that 
 \[
 (L^tf)^2\geq 2\eta(t)L^tf-\eta^2(t)
 \]
 taking into account assumption \eqref{assumption 2} and the above computations, we get the desired result.
  \begin{align*}
 (L^t-\pt)F+2\Gamma^t(f,F)&\geq(-2\alpha(t)k_1-\frac{\kappa}{\var}-\frac{4a\alpha(t)\eta(t)}{n})\Gamma^t(f)\\
 &+(2\rho_2-2\beta-b'(t))\Gamma^{\cV}(f)+(\alpha'(t)+\frac{4a\alpha(t)\eta(t)}{n})f_t\\
 &+\phi'(t)-\frac{\alpha(t)n}{2c}\max\{k_1^2,k_2^2\}
 \end{align*} 
  \end{proof}
  
  \begin{rem}\label{diff ine F}
  In the above lemma, if we choose
  \begin{align*}
  &-2\alpha(t)k_1-\frac{\kappa}{\var}-\frac{4a\alpha(t)\eta(t)}{n}=\frac{d'}{d}\\
  &2\rho_2-2\beta-b'(t)=\frac{d'}{d}b(t)\\
  &\alpha'(t)+\frac{4a\alpha(t)\eta(t)}{n}=-\frac{d'}{d}\alpha(t)\\
  &\phi'(t)-\frac{\alpha(t)n}{2c}\max\{k_1^2,k_2^2\}=-\frac{d'}{d}\phi(t)
  \end{align*}
  we can further get 
  \begin{align}\label{Qian version}
  (L^t-\pt)F+2\Gamma^t(f,F)\geq \frac{d'}{d}F
  \end{align}
  this ensures that we can assume $\Psi F>0$ in the proof of the following theorem, since otherwise $F<0$ will give us the result directly, see \cite{Qian13} for details.
  \end{rem}
 For a given $\mathcal C^1$ function $d(t):[0,\infty)\rightarrow [0,\infty)$ such that $d(0)=0,\quad \lim_{t\rightarrow 0}\frac{d(t)}{d'(t)}=0,\quad\frac{d'}{d}>0,\quad \frac{\int_0^td(s)ds}{d(t)}>0$ and some integrable conditions \cite{Qian13}, we can prove the Li-Yau type inequality Theorem \ref{diff harnack} for the positive solution for  heat equation associated with the time dependent horizontal Laplacian under transverse Ricci flow. This theorem is a generalization of \cite{Qian13}[Theorem 2.1] and \cite{BCP10}[Theorem 2.7].

 \begin{proof}{Proof of theorem \ref{diff harnack}}
 We follow the idea from Q. Zhang \cite{Zhang06} and X. Cao \cite{BCP10}, we still use the same notation as in the previous lemma for $f$ and $F$. Let us pick $\tau\in (0,T]$ and fix $\bar\Psi(x,t)$ satisfyting the conditions of Lemma \ref{Lemma Psi}. Define $\psi:\M\times{0,T}\rightarrow \R$ by setting 
 \begin{align}\label{Psi function}
 \Psi(x,t)=\bar\psi(dist(x,x_0,t),t)=\bar\psi(r_{\var}(x),t)
 \end{align}
 where we denote $r_{\var}(x)=dist(x,x_0,t)$ as the Riemannian distance from $x$ to a fixed point $x_0$ at time $t$ associate with the metric $g_{\var}(t)$. We will show our estimate at $(x,\tau)$ for $x\in\M$ such that $r_{\var}(x)<\frac{\rho}{2}.$ Our strategy is the same as in \cite{BCP10} to estimate $(\pt-L^t)(\Psi F)$ and analyze the result at a point where the function $\psi F$ attains its maximum. Lemma \ref{Lemma estimate F} and some direct computations imply
 \begin{align*}
 (L^t-\pt)(\Psi F)&\geq -2\Gamma^{t}(f,\Psi F) +2F\Gamma^t(f,\Psi)+(-2\alpha(t)k_1-\frac{\kappa}{\var}-\frac{4a\alpha(t)\eta(t)}{n})\Gamma^t(f)\Psi\\
 &+(2\rho_2-2\beta-b'(t))\Gamma^{\cV}(f)\Psi+(\alpha'(t)+\frac{4a\alpha(t)\eta(t)}{n})f_t\Psi\\
 &+(\phi'(t)-\frac{\alpha(t)n}{2c}\max\{k_1^2,k_2^2\})\Psi\\
& +2\frac{\Gamma^t(\Psi,\Psi F)}{\psi}-2\frac{\Gamma^t(\Psi)F}{\Psi}+F(L^t\Psi)-F(\pt \Psi)
 \end{align*}
 the inequality holds in the part of $B_{\rho,T}$ where $\Psi$ is smooth and strictly positive. Let $(x_1,t_1)$ be a maximum point for the fucntion $\Psi F$ in the set $\{(x,t)\in \M\times {0,\tau}| r_{\var}(x)\leq \rho\}.$ We may assume $(\Psi F)(x_1,t_1)>0$ according to Remark \ref{diff ine F}. Since $(x_1,t_1)$ is a maximum point, the formulas $L^t(\Psi F)(x_1,t_1)\leq 0,\quad \nabla (\Psi F)(x_1,t_1)=0$, and $(\psi F)_t\geq 0$ hold true. We thus get 
  \begin{align*}
0&\geq 2F\Gamma^t(f,\Psi)+(-2\alpha(t)k_1-\frac{\kappa}{\var}-\frac{4a\alpha(t)\eta(t)}{n})\Gamma^t(f)\Psi\\
 &+(2\rho_2-2\beta-b'(t))\Gamma^{\cV}(f)\Psi+(\alpha'(t)+\frac{4a\alpha(t)\eta(t)}{n})f_t\Psi\\
 &+(\phi'(t)-\frac{\alpha(t)n}{2c}\max\{k_1^2,k_2^2\})\Psi\\
&-2\frac{\Gamma^t(\Psi)F}{\Psi}+F(L^t\Psi)-F(\pt \Psi)
 \end{align*}
 according to Lemma \ref{Lemma Psi}
 \begin{align}\label{est: grad Psi}
 -\frac{|\nabla \psi|^2}{\psi}\geq -\frac{C_{1/2}^2}{\rho^2}
 \end{align}
since our transverse Ricci flow preserves the totally geodesic foliation structure, by the sub-laplacian comparison theorem \cite{BGKT17}, we have (Corollary 2.10 \cite{BGKT17})
\begin{align*}
L^tr_{\var}(x)\leq F(r_{\var},\gamma)=
\begin{cases}
\sqrt{n\kappa_{\var}}\cot(\sqrt{\frac{\kappa_{\var}}{n}}r_{\var}(x)), & \text{if} ~\kappa_{\var}>0,\\
\frac{n}{r_{\var}(x)},& \text{if}~ \kappa_{\var}=0,\\
\sqrt{n|\kappa_{\var}(\gamma)|}\coth(\sqrt{\frac{|\kappa_{\var}|}{n}}r_{\var}(x)),& \text{if} ~\kappa_{\var}<0.
\end{cases}
\end{align*}
where 
\[
\kappa_{\var}=\min\{\rho_1-\frac{\kappa}{\var},\frac{\rho_2}{\var}\}
\]
so we have the estimate
\begin{align}\label{est: L^t Psi}
L^t\psi\geq -\frac{C_{1/2}}{\rho^2}|\nabla_{\cH}r_{\var}|_{\cH}-\frac{C_{1/2}\psi^{1/2}}{\rho}F(r_{\var,\gamma})\geq -\frac{d_1}{\rho^2}-\frac{d_1\psi^{1/2}}{\rho}F(r_{\var,\gamma})
\end{align}
at the point $(x_1,t_1)$ with $d_1$ a positive constant depending on $n$. There existes $\bar C>0$ such that the inequality 
\begin{align}\label{est: pt Psi}
-\frac{\partial \Psi}{\partial t}\geq -\frac{\bar C\Psi^{1/2}}{\tau}-C_{1/2}\bar k\Psi^{1/2}\quad \bar k=\max\{k_1,k_2\}
\end{align}
this is the same as $(2.6),(2.7),(2.8)$ in \cite{BCP10}. Combining with \eqref{est: grad Psi}, \eqref{est: L^t Psi} and \eqref{est: pt Psi}, we get
\begin{align*}
0&\geq -2F|\nabla f||\nabla \Psi|+(-2\alpha(t)k_1-\frac{\kappa}{\var}-\frac{4a\alpha(t)\eta(t)}{n})\Gamma^t(f)\Psi\\
 &+(2\rho_2-2\beta-b'(t))\Gamma^{\cV}(f)\Psi+(\alpha'(t)+\frac{4a\alpha(t)\eta(t)}{n})f_t\Psi\\
 &+(\phi'(t)-\frac{\alpha(t)n}{2c}\max\{k_1^2,k_2^2\})\Psi\\
&+d_2(-\frac{1}{\rho^2}-\frac{\psi^{1/2}}{\rho}F(r_{\var},\gamma)-\frac{\Psi^{1/2}}{\tau}-\bar k\Psi^{1/2})F
 \end{align*}
where we denote $d_2=\max\{3d_1,C_{1/2},3C_{1/2}^2,\bar C\}$. Multiply $t\Psi$, use the estimate for $|\nabla \Psi|$ and the fact $2xy\leq x^2 c+\frac{y^2}{c}$, taking $c=\Psi$ we get
\begin{align*}
0&\geq -\frac{t_1d_2}{\rho\Psi}(F\Psi)^2+(-2\alpha(t)k_1-\frac{\kappa}{\var}-\frac{4a\alpha(t)\eta(t)}{n}-\frac{t_1d_2}{\rho})\Gamma^t(f)\Psi^2\\
 &+(2\rho_2-2\beta-b'(t))\Gamma^{\cV}(f)\Psi^2+(\alpha'(t)+\frac{4a\alpha(t)\eta(t)}{n})f_t\Psi^2\\
 &+(\phi'(t)-\frac{\alpha(t)n}{2c}\max\{k_1^2,k_2^2\})\Psi^2\\
&+d_2(-\frac{1}{\rho^2}-\frac{\psi^{1/2}}{\rho}F(r_{\var},\gamma)-\frac{\Psi^{1/2}}{\tau}-\bar k\Psi^{1/2})F\Psi
 \end{align*}
 Now we choose funciton $d(t)$ such that 
 \begin{align*}
  &-2\alpha(t)k_1-\frac{\kappa}{\var}-\frac{4a\alpha(t)\eta(t)}{n}-\frac{t_1d_2}{\rho}=\frac{d'}{d}\\
  &2\rho_2-2\beta-b'(t)=\frac{d'}{d}b(t)\\
  &\alpha'(t)+\frac{4a\alpha(t)\eta(t)}{n}=-\frac{d'}{d}\alpha(t)\\
  &\phi'(t)-\frac{\alpha(t)n}{2c}\max\{k_1^2,k_2^2\}=-\frac{d'}{d}\phi(t)
  \end{align*}
  we can get 
   \[
 b(t)=\frac{2(\rho_2-\beta)\int_0^td(s)ds}{d(t)},\quad \alpha(t)=\frac{e^{2k_1t}\int_0^te^{-2k_1s}d(s)(\frac{d'(s)}{d(s)}+\frac{\kappa}{\var}+\frac{t_1d_2}{\rho})ds}{d(t)}
 \]
 \[
 \phi(t)=\frac{\alpha(t)n\max\{k_1^2,k_2^2\}\int_0^td(s)ds}{2cd(t)},\quad \eta(t)=-\frac{nk_1}{2a}-\frac{n}{4a\alpha(t)}(\frac{d'(t)}{d(t)}+\frac{\kappa}{\var}+\frac{t_1d_2}{\rho}) )
 \]
which implies 
\begin{align*}
0&\geq -\frac{t_1d_2}{\rho\Psi}(F\Psi)^2-\frac{d'}{d}F\Psi^2
&+d_2(-\frac{1}{\rho^2}-\frac{\psi^{1/2}}{\rho}F(r_{\var},\gamma)-\frac{\Psi^{1/2}}{\tau}-\bar k\Psi^{1/2})F\Psi
 \end{align*}
 by completing on the squares, 
 we get 
 \begin{align}\label{est: Psi F}
 \Psi F\leq \frac{1}{4}(\frac{1}{\rho}+F(r_{\var},\gamma)+\frac{\rho}{\tau}+\rho \bar k-{\frac{\rho d'\Psi(t_1)}{d_2t_1 d}})
 \end{align}
 we denote 
 \[
 B=\frac{1}{4}(\frac{1}{\rho}+F(r_{\var},\gamma)+\frac{\rho}{\tau}+\rho \bar k-{ \frac{\rho d'\Psi(t_1)}{d_2t_1 d}})
 \]
 recall that $\Psi(x,\tau)=1$ whenever $r_{\var}(x)<\frac{\rho}{2}$ and $(x_1,t_1)$ is a maximum point for $\Psi F$ int the set  $\{(x,t)\in \M\times (0,\tau]| r_{\var}(x)\leq \rho\}.$ Hence 
 \begin{align*}
 F(x,\tau)&=(\Psi F)(x,\tau)\leq (\Psi F)(x_1,t_1)\\
 &\leq \frac{1}{4}(\frac{1}{\rho}+F(r_{\var,\gamma})+\frac{\rho}{\tau}+\rho \bar k-{\frac{\rho d'(t_1)\Psi(t_1)}{d_2t_1 d(t_1)}})
 \end{align*}
 Recall that 
  \begin{align*}
 F=\Gamma^t(f)+b(t)\Gamma^{\cV}(f)-\alpha(t)f_t-\phi(t).
 \end{align*}
 \[
 F(x,\tau)\leq B
 \]
 gives us the desired result.
 \end{proof}
 \subsection{Parabolic Harnack inequality and heat kernel upper bound.}

 Before we prove the theorem \ref{parabolic harnack}, let's estimate $\alpha(t)$ first.
 \begin{rem}\label{alpha t estimate}
  Recall that, we have $C=(2k_1+\frac{\kappa}{\var}+\frac{t_1d_2}{\rho})$,
  \[
 \alpha(t)=C[ \frac{e^{2k_1t}}{4k_1^3t^2}-(\frac{1}{2k_1}+\frac{1}{2k_1^2t}+\frac{1}{4k_1^3t^2})]+1\leq C\frac{e^{2k_1t}}{4k_1^3t^2}
 \]
 since $1-(2k_1+\frac{\kappa}{\var}+\frac{t_1d_2}{\rho})(\frac{1}{k_1}+\frac{1}{2k_1^2t}+\frac{1}{64k_1^3t^2})$ is monotone increasing on $(0,T]$, but its value is negative.
 Note that $\frac{e^{2k_1t}}{4k_1^3t^2}$ is decreasing on $[0,\sqrt{\frac{1}{k_1}})$ and increasing on $[\sqrt{\frac{1}{k_1}},\infty)$. 
 so we have
 \[
 \alpha(t)\leq  \hat \alpha(k_1,T)=:\begin{cases}
 C\frac{e}{2k_1},& \text{if}~ k_1\leq \frac{1}{T^2},\\
C \max\{ \frac{e}{2k_1}, \frac{e^{2k_1T}}{4k_1^3T^2} \}, & \text{if}~ k_1\geq \frac{1}{T^2}.
 \end{cases}
 \]
 As for the lower bound for $\alpha(t)$, we make Taylor expansion for $e^{2k_1t}$ up to $4^{th}$ order, we then get
 \begin{align*}
 \alpha(t)&=1+\frac{C}{3}t+o(t)\geq 1 ,\quad
 \text{for} ~t\in (0,T], \text{and}~ C=(2k_1+\frac{\kappa}{\var}+\frac{t_1d_2}{\rho})
 \end{align*}
 thus we have $\frac{1}{\alpha(t)}\leq 1$ for $\forall t\in (0,T]$.
 \end{rem}
 \begin{proof}{Proof of theorem \ref{parabolic harnack}}
 The proof for this theorem is quite traditional. Similar proof can be found in \cite{BaudoinGarofalo14},\cite{BCP10} and \cite{Qian13}. We fix two points $(x,s),(y,t)\in \M \times (0,T)$ with $s<t$. Let $\gamma(\tau),$ $0\leq \tau\leq T$ be a subunit path such that $\gamma(0)=y,\gamma(T)=x$, consider the path in $\M \times (0,T)$ defined by 
 \[
 m(\tau)=(\gamma(\tau),t+\frac{s-t}{T}\tau)
 \]
 such that $m(0)=(y,t),m(T)=(x,s)$. We have 
 \begin{align*}
 \log \frac{u(x,s)}{u(y,t)}&=\int_0^T\frac{d}{d\tau}\log u(m(\tau))d\tau\\
 &\leq \int_0^T |\nabla_{\cH}\log u||\frac{\partial \gamma}{\partial \tau}|-\frac{t-s}{T}\frac{\partial \log u}{\partial t}(m(\tau))d\tau
 \end{align*} 
 using Corollary \ref{cor: diff harnack}
  \begin{align*}
 \log \frac{u(x,s)}{u(y,t)}&\leq \int_0^T [|\nabla_{\cH}\log u||\frac{\partial \gamma}{\partial \tau}|-  \frac{t-s}{T} \frac{1}{\alpha(\tau)}\Gamma^{\tau}(\log u) +\frac{t-s}{T\alpha(\tau)}A+\frac{t-s}{T\alpha(\tau)}\frac{\rho}{t+\frac{s-t}{T}\tau} ]d\tau\\
 &\leq \int_0^T \frac{\alpha(\tau)}{4}\frac{T}{t-s}|\frac{\partial \gamma}{\partial \tau}|^2d\tau+\int_0^T\frac{t-s}{T\alpha(\tau)}Ad\tau+\int_s^t\frac{\rho}{4\tau}\frac{1}{\alpha(\frac{\tau-t}{s-t}T)}d\tau\\
 &\leq \frac{T\hat \alpha(T,k_1)}{4(t-s)}\int_0^T|\frac{\partial \gamma}{\partial \tau}|^2d\tau +A(t-s)+\frac{\rho}{4}\log(\frac{t}{s})
 \end{align*}
  wtih 
  \[
 A=\frac{n\bar k}{2}+\frac{1}{4}(\frac{1}{\rho}+F(r_{\var},\gamma)+\rho\bar k), \bar k=\max\{k_1,k_2\}
 \]
 the last inequality follows from the Remark \ref{alpha t estimate}. We eventually get
 \begin{align*}
 u(x,s)\leq u(y,t)(\frac{t}{s})^{\frac{\rho}{4}}\exp\left(\frac{T\hat \alpha(k_1,T)}{4(t-s)}r_{}^2(x,y)+A(t-s) \right)
 \end{align*}
and $r_{}$ is the subunit length of $\gamma$, namely the Carnot-Carath\'eodory distance $d_{cc}$ between $x$ and $y$ induced by the subelliptic operator $L^t$.
 \end{proof}
 \begin{proof}{Proof of corollary \ref{kernel upper bound}}
The proof is similar to Corollary $7.2$ in \cite{BaudoinGarofalo14}.
\end{proof}

\section*{Acknowledgement}
The author would like to thank his thesis advisor Prof. Fabrice Baudoin for his constant support and encouragement, as well as for all his advices and discussions along the project. The author would also like to thank Prof. Ovidiu Munteanu for some helpful discussions.

\end{document}